\documentclass[11pt
]{amsart}

\usepackage{hyperref}
\hypersetup{bookmarksdepth=3}

\usepackage{geometry}
\usepackage{amsmath,amssymb}
\usepackage{
mathrsfs}
\usepackage{esint}

\usepackage{tensor}

\usepackage{amssymb,amsmath,amsthm}

\newtheorem{theorem}{Theorem}

\newtheorem{definition}[theorem]{Definition}

\theoremstyle{remark}\newtheorem{remark}[theorem]{Remark}

\newtheorem{proposition}[theorem]{Proposition}

\usepackage{graphicx}

\begin{document}

{\let\thefootnote\relax\footnote{Date: 1st September 2016.

\textcopyright 2016 by the authors. Faithful reproduction of this article, in its entirety, by any means is permitted for noncommercial purposes.}}

\title{on existence of global solutions of the one-dimensional cubic NLS for initial data in the modulation space $M_{p,q}(\mathbb R)$. }

\author{L. Chaichenets}
\address{leonid chaichenets, department of mathematics, institute for analysis, karlsruhe institute of technology, 76128 karlsruhe, germany }
\email{leonid.chaichenets@kit.edu}

\author{D. Hundertmark}
\address{dirk hundertmark, department of mathematics, institute for analysis, karlsruhe institute of technology, 76128 karlsruhe, germany }
\email{dirk.hundertmark@kit.edu}

\author{P. Kunstmann}
\address{peer kunstmann, department of mathematics, institute for analysis, karlsruhe institute of technology, 76128 karlsruhe, germany }
\email{peer.kunstmann@kit.edu}

\author{N. Pattakos}
\address{nikolaos pattakos, department of mathematics, institute for analysis, karlsruhe institute of technology, 76128 karlsruhe, germany }
\email{nikolaos.pattakos@gmail.com}

\begin{abstract}
{We prove global existence for the one-dimensional cubic non-linear Schr\"odinger equation in modulation spaces $M_{p, p'}$ for $p$ sufficiently close to $2$. In contrast to known results, \cite{KAT} and \cite{BH}, our result requires no smallness condition on initial data. The proof adapts a splitting method inspired by work of Vargas-Vega and Hyakuna-Tsutsumi to the modulation space setting and exploits polynomial growth of the free Sch\"odinger group on modulation spaces. }
\end{abstract}

\maketitle
\pagestyle {myheadings}

\begin{section}{introduction and main result}
\markboth{\normalsize L. Chaichenets, D. Hundertmark, P. Kunstmann and N. Pattakos }{\normalsize  Global Existence in $M_{p,q}(\mathbb R)$}
In this paper we are going to investigate the global wellposedness theory for the one-dimensional cubic non-linear Schr\"odinger equation

\begin{equation}
\label{maineq}
\begin{cases} iu_{t}+u_{xx}\pm|u|^{2}u=0 &,\ (t,x)\in\mathbb R^{2}\\
u(0,x)=u_{0}(x) &,\ x\in\mathbb R\\
\end{cases}
\end{equation}
where $u_{0}$ lies in a modulation space $M_{p,q}$.
Local wellposedness in modulation spaces $M_{p,1}$ has been shown in \cite{BO} and there are some global existence results under smallness conditions on the initial value in \cite{KAT} and \cite{BH} (see also \cite{RUU} and \cite{WZCZ}). However, in dimension one the cubic nonlinearity is not covered by these results. On the other hand, it is well known that if $u_{0}\in L^{2}$ then the initial value problem (\ref{maineq}) is globally well posed and that the $L^{2}$ norm is conserved. This was proved in \cite{Tsu}. In \cite{VV} it was proved that under some weaker assumptions on $u_{0}$ we can still obtain global existence results even if the $L^{2}$ norm of $u_{0}$ is infinite. The idea is to split the initial data $u_{0}$ between two suitable function spaces and solve in each of them a different NLS and then combine the solutions to get a function that solves problem (\ref{maineq}). This idea was exploited further in \cite{HYTSU} for $\hat{u_{0}}\in L^{p'}$ with $p$ close to $2$. The method of splitting itself goes back to \cite{BOU} at least.

Before we state our result and the proper definition of $M_{p,q}$ let us denote by $S(\mathbb R)$ the set of all Schwartz functions and by $S'(\mathbb R)$ its dual space. Fix $s\in\mathbb R$ and $0<p, q\leq \infty$. Then

\begin{equation}
\label{def0}
M^{s}_{p,q}=\Big\{f\in S'(\mathbb R):\|f\|_{M^{s}_{p,q}}<\infty\Big\}
\end{equation}
where the quantity $\|f\|_{M^{s}_{p,q}}$ is defined as

\begin{equation}
\label{def1}
\|f\|_{M^{s}_{p,q}}=\Big\|\xi\mapsto (1+|\xi|)^{s}\|(V_{g}f)(\cdot, \xi)\|_{p}\Big\|_{q}
\end{equation}
and $V_{g}f$ is the short time Fourier transform of the function $f$ with window $g\in S(\mathbb R)\setminus\{0\}$, that is

\begin{equation}
\label{def2}
V_{g}f(x,\xi)=\int_{\mathbb R}e^{-i\xi y}\overline{g(y-x)}f(y)\ dy.
\end{equation}
It can be shown that different choices of the window function $g$ lead to equivalent norms on $M^{s}_{p,q}$. These spaces were first introduced in \cite{FEI} and many of their properties such as embeddings in other known function spaces and equivalent expressions for their norm can be found in \cite{BH} where it is also proved that for $s_{1}\geq s_{2}$, $0<p_{1}\leq p_{2}$ and $0<q_{1}\leq q_{2}$ we have the relation $M^{s_{1}}_{p_{1}, q_{1}}\subset M^{s_{2}}_{p_{2}, q_{2}}$. Since their introduction, modulation spaces have become canonical for both time-frequency and phase-space analysis. They provide an excellent substitute in estimates that are known to fail on Lebesgue spaces. 

Every time we write $\|f\|_{p}$ or $\|f\|_{L^{p}}$ we mean the usual $p$-norms in the Lebesgue spaces $L^{p}(\mathbb R)$. In addition, for a given interval $I\subset\mathbb R$ we use the notation $\|f\|_{L^{p}_{I}}$ for the $L^{p}$ norm of $f$ over $I$. We also denote by $M_{p,q}$ the modulation space $M_{p,q}^{0}$. 

To state our main result we need to define the following set of functions. 

\begin{definition}
\label{defmain}
For a given $r>1$, $0<\alpha<1$ and $c_{0}>0$ we define the set $S_{\alpha, c_{0}}^{r}$ to be the collection of all $u\in L^{2}+M_{r,\frac{r}{r-1}}$ such that for every $N\in\mathbb R_{+}$ there are functions $\phi^{N}\in L^{2}$ and $\psi^{N}\in M_{r,\frac{r}{r-1}}$ with the properties $u=\phi^{N}+\psi^{N}$ and $\|\phi^{N}\|_{2}\leq c_{0}N^{\alpha}$ and $\|\psi^{N}\|_{M_{r, \frac{r}{r-1}}}\leq \frac{c_{0}}{N}$. 
\end{definition}

\begin{remark}
\label{rem}
For any $r>2$ and $\alpha>0$ define the number $p=p(r,\alpha)=\frac{2r+2r\alpha}{r+2\alpha}\in(2,r)$. Then, we have the relation $M_{p,p'}\subset \bigcup_{c_{0}\in\mathbb R_{+}} S_{\alpha,c_{0}}^{r}$ which implies that the sets $S_{\alpha,c_{0}}^{r}$ are non empty. To see this we use Theorem $6.1$ D from \cite{FEI} which shows that $M_{p,p'
}$ can be obtained as the complex interpolation space between $L^{2}=M_{2,2}$ and $M_{r,r'}$, that is $M_{p,p'}=[L^{2},M_{r,r'}]_{\theta}$ for $\theta=\frac{\alpha}{\alpha+1}$. Next, we use Proposition $2.10$ from \cite{LUN} which states that $M_{p,p'}$ can be continuously embedded in the real interpolation space $(L^{2},M_{r,r'})_{\theta,\infty}$ and then we take a look at the $K$ functional which induces a norm to $(L^{2},M_{r,r'})_{\theta,\infty}$ by the formula

$$\|u\|_{\theta,\infty}=\sup_{t>0}\Big(t^{-\theta}\inf_{u=\phi+\psi}\Big[\|\phi\|_{2}+t\|\psi\|_{M_{r,r'}}\Big]\Big).$$
It is easy to see that for a given $N\in\mathbb R_{+}$, we must have $N^{-\alpha}=t^{-\theta}$ and $N=t^{1-\theta}$ or equivalently $t=N^{\alpha+1}$ and $\theta=\frac{\alpha}{\alpha+1}$.
\end{remark}

Our main theorem is as follows:

\begin{theorem}
\label{main}
Suppose that $u_{0}\in S_{\alpha, c_{0}}^{r}$ where we have $r\in(3,4]$, $c_{0}>0$ and $\alpha\in (0,\frac{6r-r^{2}}{22r^{2}-38r+12})$. Then, the Cauchy problem (\ref{maineq}) has a unique global solution $u$ that can be written as a sum of two functions $v, w$ that lie in the spaces

$$v\in L^{\frac{4r}{r-2}}_{loc}(\mathbb R: L^{r}(\mathbb R))\cap L^{\infty}_{loc}(\mathbb R: L^{2}(\mathbb R))$$
and

$$w\in L^{Q}_{loc}(\mathbb R:M_{r,\frac{r}{r-1}}(\mathbb R)),$$
for any sufficiently large $Q$. 
\end{theorem}

\begin{remark}
In the literature the only global existence results for NLS (\ref{maineq}) with initial data in a modulation space require the modulation norm to be small. See \cite{KAT} and \cite{BH} (also \cite{RUU} and \cite{WZCZ}) for more details. As we shall see our approach works with no restrictions on the modulation norm of the initial condition.
\end{remark}

\begin{remark}
An easy computation shows that the maximum of the function $p(r,\alpha)$, defined in Remark \ref{rem}, over the domain $\{(r,\alpha)\in\mathbb R^2:3\leq r\leq 4, 0<\alpha\leq\frac{6r-r^{2}}{22r^{2}-38r+12}\}$ is attained at the corner $(3,\frac3{32})$ and gives the maximum value of $p=p(3,\frac3{32})=\frac{35}{17}$. Of course, in Theorem \ref{main}, $r$ lies in the interval $(3,4]$ which means that the space $M_{\frac{35}{17},\frac{35}{18}}$ is not covered by our main Theorem. 
\end{remark}

\begin{remark}
Let us make some comments about the key ingredients of the proof of Theorem \ref{main}. We follow closely the calculations presented in \cite{HYTSU} by Hyakuna-Tsutsumi where the initial data $u_{0}$ was split as a sum of a good function $\phi\in L^2$ and a bad function $\psi$ depending on a (large) parameter $N$. One has a global solution for the NLS with initial value $\phi$ and a global solution for the linear evolution with initial value $\psi$. The nonlinear interaction is shown to exist for a small time $\delta_N$ and takes values in $L^2$. This step can be repeated sufficiently many times (depending on $N$). In \cite{HYTSU}, many of the quantities were conserved from one step to the next, in particular, the norm for the linear evolution of $\psi$.
  This is not the case here since we deal with the $M_{r,r'}$ norm and have polynomial growth. Consequently, we have to consider in each step $k+1$ a smaller time existence interval $\delta_{N}^{(k+1)}<\delta_{N}^{(k)}$, and these have to be chosen in such a way that the (finite) series $\sum_k\delta_{N}^{(k)}$ diverges to infinity as $N\to\infty$. This turns out to be possible and is crucial for our global existence result.
The restriction on $\alpha$ (depending on $r\in(3,4]$) in Theorem \ref{main} comes precisely from the divergence condition on $\sum_k\delta_{N}^{(k)}$ for $N\to\infty$.
\end{remark}

\begin{remark}
Theorem \ref{main} remains true in higher dimensions and with nonlinearities of the form $|u|^{p-1}u$, $1<p<1+\frac4{d}$ with proper adjustments in the range of $r$ and $\alpha$. We concentrate on the one-dimensional cubic NLS for presentation reasons.
\end{remark}

\end{section}

\begin{section}{preliminaries}

From \cite{BH} it is known that for any $1<p\leq\infty$ the space $M_{p,1}$ embeds into $L^{\infty}\cap L^{p}$. This means that $M_{\infty, 1}\hookrightarrow  L^{\infty}$ and since $M_{2,2}=L^{2}$ we obtain by interpolation the embedding $M_{p, p'}\hookrightarrow L^{p}$, for $2\leq p\leq \infty$. Therefore, in the following we will use the embedding $M_{r,\frac{r}{r-1}}\hookrightarrow L^{r}$ which implies that there is a constant $c_{E}>0$ so that the inequality

\begin{equation}
\label{embed}
\|f\|_{r}\leq c_{E}\|f\|_{M_{r,\frac{r}{r-1}}}
\end{equation}
holds for all $f\in L^{r}$. Another fact about modulation spaces that we are going to use is inequality

\begin{equation}
\label{modin}
\|e^{it\partial_{x}^{2}}f\|_{M_{p,q}}\leq c(1+|t|)^{\frac12-\frac1{p}}\|f\|_{M_{p,q}}
\end{equation}
which holds for all $f\in M_{p,q}$ and $t\in\mathbb R$, where $c>0$ is independent of $f$ and the time $t$. Since our $p=r$ we have that there is a universal constant $c_{I}>0$ such that the following is true

\begin{equation}
\label{modinin}
\|e^{it\partial_{x}^{2}}f\|_{M_{r,\frac{r}{r-1}}}\leq c_{I}(1+|t|)^{\frac12-\frac1{r}}\|f\|_{M_{r,\frac{r}{r-1}}}.
\end{equation}

Before we proceed with the proof of our main Theorem, which is in Section \ref{pr}, we need to state and prove some preliminary results. Throughout the paper we will use the notation $G(v,w)=|v+w|^{2}(v+w)-|v|^{2}v$ and $\tilde{G}(v,w_{1},w_{2})=G(v,w_{1})-G(v,w_{2})$ for $v, w, w_{1}, w_{2}\in\mathbb C$. 

A pair of numbers $(r, q)$ is called admissible if $2\leq r\leq 6$, $q>2$ and they satisfy the equation

$$\frac1{q}+\frac1{2r}=\frac14.$$
For such pairs the following proposition is true.

\begin{proposition}
\label{admi}
Suppose that $v$ solves the initial value problem (\ref{maineq}) with initial data $\phi\in L^{2}$ and that $(r, q)$ is an admissible pair. Then there are constants $k_{1}, k_{2}>0$ independent of $\phi$ with the property that 

$$\|v\|_{L^{q}_{I_{\delta}}L^{r}}\leq k_{1}\|\phi\|_{2},$$
for all $\delta\in[0, (k_{2}\|\phi\|_{2})^{-4}]$, where $I_{\delta}=[0,\delta]$.
\end{proposition}
This can be proved by interpolating between the estimate for the $L^{\infty}_{I_{\delta}}L^{2}$ norm, which is obvious since the solution of NLS (\ref{maineq}) has conserved $L^{2}$ norm, and the estimate for the $L^{6}_{I_{\delta}}L^{6}$ norm which can be proved by showing that the set

$$\Big\{\delta\in[0, (k_{2}\|\phi\|_{2})^{-4}]:\|v\|_{L^{6}_{I_{\delta}}L^{6}}\leq k_{1}\|\phi\|_{2}\Big\}$$
is a nonempty, open and closed subset of $[0, (k_{2}\|\phi\|_{2})^{-4}]$. To see this write 

$$v=e^{it\partial_{x}^{2}}\phi\pm i\int_{0}^{t}e^{i(t-\tau)\partial_{x}^{2}}[|v|^{2}v]\ d\tau,$$
and estimate each of the two summands in the $L^{6}_{I_{\delta}}L^{6}$ norm. For the evolution part $e^{it\partial_{x}^{2}}\phi$ the result is known, see \cite{HTT} (Theorem $1.4$), and for the convolution integral part an application of H\"older's inequality implies the desired result. We refer to \cite{HTT} for more details, where a similar result was proved for Lorentz type norms. 
 
Next we define the triangles

$$\widehat{T}_{1}=\Big\{(x,y)\in\Big(0,\frac12\Big)^{2}:y<-x+\frac12\Big\}\cup\Big\{\Big(\frac12, 0\Big)\Big\}$$
and

$$\widehat{T}_{2}=\Big\{(x,y)\in\Big(\frac12,1\Big)^{2}:y>-x+\frac32\Big\}.$$
Then we have the following proposition which can be found in \cite{HYTSU} and \cite{KA}:

\begin{proposition}
\label{inter}
Suppose that $(\frac1{r}, \frac1{q})\in\widehat{T}_{1}$ and $(\frac1{p},\frac1{\gamma})\in\widehat{T}_{2}$ and

$$\frac2{\gamma}+\frac1{p}=2+\frac2{q}+\frac1{r}.$$
Then there is a constant $C>0$ that depends on $q, r$ such that the estimate

$$\Big\|\int_{0}^{t}e^{i(t-\tau)\partial_{x}^{2}}F(\cdot, \tau)\ d\tau\Big\|_{L^{q}_{I_{T}}L^{r}}\leq C\|F\|_{L^{\gamma}_{I_{T}}L^{p}}$$
is valid for any $T>0$ and $F\in L^{\gamma}_{I_{T}}L^{p}$.
\end{proposition}

Special cases of the following two propositions can be found again in \cite{HYTSU} but since we need them in a more general setting we present their proofs, too.

\begin{proposition}
\label{Tsu1}
Let $q>3$ and $\max\{3,\frac{2q}{q-2}\}<r<\min\{q, \frac{4q}{q-2}\}$. Then there is a constant $C=C_{q,r}>0$ such that for all $T>0$ the quantity

$$\Big\|\int_{0}^{t}e^{i(t-\tau)\partial_{x}^{2}}\tilde{G}(v,w_{1},w_{2})\ d\tau\Big\|_{L^{q}_{I_{T}}L^{r}}$$
is bounded above by 

$$C\cdot\Big[T^{\frac12}\|v\|^{2}_{L^{\frac{4r}{r-2}}_{I_{T}}L^{r}}\|w_{1}-w_{2}\|_{L^{q}_{I_{T}}L^{r}}+T^{\frac{2rq-2q-4r}{2rq}}(\|w_{1}\|_{L^{q}_{I_{T}}L^{r}}^{2}+\|w_{2}\|^{2}_{L^{q}_{I_{T}}L^{r}})\|w_{1}-w_{2}\|_{L^{q}_{I_{T}}L^{r}}\Big].$$
\end{proposition}

\begin{proof}
Since the pair $(\frac1{r}, \frac1{q})$ belongs to the triangle $\widehat{T}_{1}$ we use Proposition \ref{inter} to estimate the norm

$$\Big\|\int_{0}^{t}e^{i(t-\tau)\partial_{x}^{2}}\tilde{G}(v,w_{1},w_{2})\ d\tau\Big\|_{L^{q}_{I_{T}}L^{r}}$$
by the expression $C_{q,r}\cdot\|\tilde{G}\|_{L^{\gamma}_{I_{T}}L^{p}}$ where the pair $(\frac1{p},\frac1{\gamma})\in\widehat{T}_{2}$ satisfies

$$\frac2{\gamma}+\frac1{p}=2+\frac2{q}+\frac1{r}$$
and $C_{q,r}$ is a positive constant. Then we bound the function $|\tilde{G}|$ pointwise by the expression $(|v|^{2}+|w_{1}|^{2}+|w_{2}|^{2})|w_{1}-w_{2}|$ and proceed with H\"older in the space norm first and then in the time norm and try to identify all the exponents that appear in the procedure as functions of the variables $q, r$ only. 

In the space norm we use the exponent $A$ with conjugate $A'$ and then in the time variable we use the exponents $B_{1}, B_{2}, B_{3}$ to arrive at the upper bound

$$\||v|^{2}|w_{1}-w_{2}|\|_{L^{\gamma}_{I_{T}}L^{p}}\leq T^{\frac1{\gamma B_{3}}}\|v\|^{2}_{L^{2\gamma B_{1}}_{I_{T}}L^{2pA}}\|w_{1}-w_{2}\|_{L^{\gamma B_{2}}_{I_{T}}L^{pA'}}.$$
Since we need $L^{\gamma B_{2}}_{I_{T}}L^{pA'}=L^{q}_{I_{T}}L^{r}$ we require $q=\gamma B_{2}$, $r=pA'$. Then the pair $(2pA, 2\gamma B_{1})$ has to be admissible which means 

$$\frac1{2\gamma B_{1}}+\frac1{4pA}=\frac14,$$
and since $B_{1}, B_{2}, B_{3}$ are H\"older exponents we need 

$$\frac1{B_{1}}+\frac1{B_{2}}+\frac1{B_{3}}=1.$$
Similarly, we have the inequality

$$\||w_{1}|^{2}|w_{1}-w_{2}|\|_{L^{\gamma}_{I_{T}}L^{p}}\leq T^{\frac1{\gamma\beta_{3}}}\|w_{1}\|^{2}_{L^{2\gamma\beta_{1}}_{I_{T}}L^{2p\alpha}}\|w_{1}-w_{2}\|_{L^{\gamma\beta_{2}}_{I_{T}}L^{p\alpha'}}$$
where $\alpha, \alpha'$ are conjugate exponents and since we need all norms to be on the space $L^{q}_{I_{T}}L^{r}$, we require the identities $2\gamma\beta_{1}=q$, $2p\alpha=r$, $\gamma\beta_{2}=q$, $p\alpha'=r$. Again, the numbers $\beta_{1}, \beta_{2}, \beta_{3}$ are H\"older exponents and this implies

$$\frac1{\beta_{1}}+\frac1{\beta_{2}}+\frac1{\beta_{3}}=1.$$
From $p\alpha'=r=pA'$ we get $\alpha=A$ and from $2p\alpha=r=pA'$ we get $2\alpha=A'$. Thus, $A=\frac32$ and $A'=3$. From there on it is easy to solve and find the following expressions for all the exponents in terms of $q, r$: 

$$p=\frac{r}3$$

$$\gamma=\frac{2qr}{2qr+2r-2q}$$

$$B_{1}=\frac{2qr+2r-2q}{q(r-2)}=2\beta_{1}$$

$$B_{2}=\frac{2rq+2r-2q}{2r}=\beta_{2}$$

$$B_{3}=\frac{2rq+2r-2q}{rq}$$

$$\beta_{3}=\frac{2qr+2r-2q}{2qr-2q-4r}.$$
Substituting we arrive at the desired upper bound. The restrictions on $q$ and $r$ arise from the fact that the pairs  $(\frac1{r}, \frac1{q})\in\widehat{T}_{1}$ and $(\frac1{p},\frac1{\gamma})\in\widehat{T}_{2}$ and that the numbers $B_{i}, \beta_{i}\in (1,\infty)$ for all $i=1, 2, 3$.  

\end{proof}

\begin{proposition}
\label{Tsu2}
Let $q>\frac43$ and $3<r<6$. Then there is a constant $c=c_{q,r}>0$ such that for all $T>0$ the quantity

$$\Big\|\int_{0}^{t}e^{i(t-\tau)\partial_{x}^{2}}\tilde{G}(v,w_{1},w_{2})\ d\tau\Big\|_{L^{\infty}_{I_{T}}L^{2}}$$
is bounded above by 

$$c\cdot\Big[T^{\frac{3qr-4r-2q}{4qr}}\|v\|^{2}_{L^{\frac{4r}{r-2}}_{I_{T}}L^{r}}\|w_{1}-w_{2}\|_{L^{q}_{I_{T}}L^{r}}+T^{\frac{5qr-6q-12r}{4rq}}(\|w_{1}\|_{L^{q}_{I_{T}}L^{r}}^{2}+\|w_{2}\|_{L^{q}_{I_{T}}L^{r}}^{2})\|w_{1}-w_{2}\|_{L^{q}_{I_{T}}L^{r}}\Big].$$
\end{proposition}

\begin{proof}
We have $(\frac12, 0)\in\widehat{T}_{1}$ and so by Proposition \ref{inter}, for all pairs $(\frac1{\tilde{p}},\frac1{\tilde{\gamma}})\in\widehat{T}_{2}$ that satisfy

$$\frac2{\tilde{\gamma}}+\frac1{\tilde{p}}=\frac52$$
we bound the expression 

$$\Big\|\int_{0}^{t}e^{i(t-\tau)\partial_{x}^{2}}\tilde{G}(v,w_{1},w_{2})\ d\tau\Big\|_{L^{\infty}_{I_{T}}L^{2}}$$
by $c_{q,r}\cdot\|\tilde{G}\|_{L^{\tilde{\gamma}}_{I_{T}}L^{\tilde{p}}}$, where $c_{q,r}$ is a positive constant. Again by estimating $|\tilde{G}|$ pointwise and using H\"older we arrive at exactly the same upper bound as in the previous proof but with exponents $\tilde{A}, \tilde{B_{1}}, \tilde{B_{2}}, \ldots$ In this case they are more easily identified as $\tilde{A}=\frac32, \tilde{A}'=3$, $\tilde{p}=\frac{r}3$

$$\tilde{\gamma}=\frac{4r}{5r-6}=2\tilde{\beta_{1}}$$

$$\tilde{B_{1}}=\frac{5r-6}{2(r-2)}=\tilde{\beta_{2}}$$

$$\tilde{B_{2}}=\frac{q(5r-6)}{4r}$$

$$\tilde{B_{3}}=\frac{q(5r-6)}{3qr-4r-2q}$$

$$\tilde{\beta_{3}}=\frac{q(5r-6)}{5qr-6q-12r}.$$
\end{proof}

\end{section}

\begin{section}{proof of theorem \ref{main}}
\label{pr}
\begin{proof}

Let us fix an $r\in(3,4]$ and we choose $Q$ large enough such that 

\begin{equation}
\label{lim}
(6r-r^{2})Q^{2}>2\alpha[(11r^{2}-19r+6)Q^{2}-14r^{2}Q-24r^{2}].
\end{equation}
This can be done since $\alpha\in(0,\frac{6r-r^{2}}{22r^{2}-38r+12})$ and

$$\lim_{Q\to\infty}\frac{(6r-r^{2})Q^{2}}{2[(11r^{2}-19r+6)Q^{2}-14r^{2}Q-24r^{2}]}=\frac{6r-r^{2}}{22r^{2}-38r+12}.$$
For this $Q$, by Proposition \ref{Tsu1} there is a constant $C_{Q,r}>0$ and by Proposition \ref{Tsu2} a constant $c_{Q,r}>0$. Define $c_{Q}=\max\{C_{Q,r}, c_{Q,r}\}$. We start by choosing a positive number $M_{0}\geq k_{2}^{4}\cdot c_{0}^{4}$ such that the following four inequalities are true

\begin{equation}
\label{first}
12\cdot c_{Q}\cdot k_{1}^{2}\cdot c_{0}^{2}\leq\sqrt{M_{0}}
\end{equation}

\begin{equation}
\label{second}
27\cdot c_{Q}\cdot c_{0}^{2}\cdot c_{E}^{2}\cdot c_{I}^{2}\cdot\Big (\frac14\Big)^{\frac2{Q}}\leq\frac{M_{0}^{\frac{rQ-Q-2r}{rQ}}}{(\frac54)^{1-\frac2{r}}}
\end{equation}

\begin{equation}
\label{third}
12\cdot c_{Q}\cdot k_{1}^{2}\cdot c_{0}^{3}\cdot c_{E}\cdot c_{I}\cdot\Big(\frac14\Big)^{\frac1{Q}}\leq\frac{M_{0}^{\frac{3Qr-4r-2Q}{4Qr}}}{(\frac54)^{\frac12-\frac{1}{r}}}
\end{equation}

\begin{equation}
\label{forth}
27\cdot c_{Q}\cdot c_{0}^{3}\cdot c_{E}^{3}\cdot c_{I}^{3}\cdot\Big(\frac14\Big)^{\frac3{Q}}\leq\frac{M_{0}^{\frac{5Qr-6Q-12r}{4rQ}}}{(\frac54)^{\frac32-\frac3{r}}}.
\end{equation}
Furthermore, for each non negative integer $k$ we let

\begin{equation}
\label{defdef}
M_{k}=(5+k)^{\frac{6(r-2)Q}{5Qr-6Q-12r}}\cdot\frac{M_{0}}{5^{\frac{6(r-2)Q}{5Qr-6Q-12r}}}
\end{equation}
and 

\begin{equation}
\label{dede}
\delta_{N}^{(k)}=\frac{M_{k}^{-1}}{N^{4\alpha}}>0.
\end{equation} 
Since $M_{k}, N$ are going to be large numbers we can always assume that $\delta_{N}^{(k)}<\frac14$. Note that for each $k$, $M_{k+1}>M_{k}\geq M_{0}$. Our goal is to start with the initial splitting of the given function $u_{0}=\phi_{0}^{N}+\psi_{0}^{N}$ and show that there is a solution $u$ of NLS (\ref{maineq}) in the interval $[0,\delta_{N}^{(0)}]$ of the form $u=v^{(0)}+w^{(0)}$ where $v^{(0)}$ and $w^{(0)}$ lie in the required spaces and such that for all $t\in[0,\delta_{N}^{(0)}]$ we have the estimate

$$\|w(t)-e^{it\partial_{x}^{2}}\psi_{0}^{N}\|_{L^{\infty}_{I_{\delta_{N}^{(0)}}}L^{2}}\leq 2\cdot\frac1{N^{1+\alpha\cdot\frac{Qr-4r-2Q}{Qr}}}.$$
Then we extend our solution to the interval $[0,\delta_{N}^{(0)}+\delta_{N}^{(1)}]$ by repeating the same procedure but for the new initial data defined as the sum of the following two functions

$$\phi_{1}^{N}(x)=v^{(0)}(\delta_{N}^{(0)}, x)\pm i\int_{0}^{\delta_{N}^{(0)}}e^{i(\delta_{N}^{(0)}-\tau)\partial_{x}^{2}}G(v^{(0)},w^{(0)})\ d\tau\in L^{2}$$
and

$$\psi_{1}^{N}(x)=e^{i\delta_{N}^{(0)}\partial_{x}^{2}}\psi_{0}^{N}(x)\in M_{r,\frac{r}{r-1}}.$$
Here it is important to point out that in order to be able to prove such a claim we must have that the $L^{2}$ norm of the new function $\phi_{1}^{N}$ is bounded from above by the quantity $2c_{0}N^{\alpha}$. Inductively, at the $(k+1)$th step we define the functions $\phi_{k+1}^{N}$ and $\psi_{k+1}^{N}$ by the formulas

\begin{equation}
\label{imp1}
\phi_{k+1}^{N}(x)=v^{(k)}(\delta_{N}^{(k)}, x)\pm i\int_{0}^{\delta_{N}^{(k)}}e^{i(\delta_{N}^{(k)}-\tau)\partial_{x}^{2}}G(v^{(k)},w^{(k)})\ d\tau\in L^{2}
\end{equation}
and 

\begin{equation}
\label{imp2}
\psi_{k+1}^{N}(x)=e^{i\delta_{N}^{(k)}\partial_{x}^{2}}\psi_{k}^{N}(x)=e^{i\sum_{i=0}^{k}\delta_{N}^{(i)}\partial_{x}^{2}}\psi_{0}^{N}(x)\in M_{r,\frac{r}{r-1}}.
\end{equation}
Due to the way these functions were chosen, we have the following estimate on the $L^{2}$ norm of $\phi_{k+1}^{N}$

\begin{equation}
\label{normm}
\|\phi_{k+1}^{N}\|_{2}\leq\|\phi_{0}^{N}\|_{2}+2(k+1) N^{-1-\alpha\cdot\frac{Qr-4r-2Q}{Qr}}\leq c_{0}N^{\alpha}+2(k+1) N^{-1-\alpha\cdot\frac{Qr-4r-2Q}{Qr}}.
\end{equation}
To make this precise, we will use induction on $k$. Let us assume that we have made $k$ steps already and that $\|\phi_{k+1}^{N}\|_{2}\leq 2c_{0}N^{\alpha}$. We want to do the $(k+1)$ step. That is, we want to solve the $1$-dimensional cubic (NLS) with initial data $\phi_{k+1}^{N}+\psi_{k+1}^{N}$ in the interval $[0,\delta_{N}^{(k+1)}]$. First we solve the initial value problem

\begin{equation}
\label{wd0}
\begin{cases} iv+v_{xx}\pm|v|^{2}v=0 &,\ (t,x)\in\mathbb R^{2}\\
v(0,x)=\phi_{k+1}^{N}(x)\in L^{2} &,\ x\in\mathbb R\\
\end{cases},
\end{equation}
from which we obtain a globally defined function $v^{(k+1)}$ such that $\|v^{(k+1)}(t,\cdot)\|_{2}=\|\phi_{k+1}^{N}\|_{2}$ for all times $t$ and such that $\|v^{(k+1)}\|_{L^{\frac{4r}{r-2}}_{I_{\delta_{N}^{(k+1)}}}L^{r}}\leq k_{1}\|\phi_{k+1}^{N}\|_{2}\leq2c_{0}k_{1}N^{\alpha}$ since the pair $(r,\frac{4r}{r-2})$ is admissible. Then we need to solve 

\begin{equation}
\label{wd}
\begin{cases} iw+w_{xx}\pm G(v^{(k+1)},w)=0 &,\ (t,x)\in\mathbb R^{2}\\
w(0,x)=\psi_{k+1}^{N}(x)\in M_{r,\frac{r}{r-1}} &,\ x\in\mathbb R\\
\end{cases}.
\end{equation}
For this we define the function space

\begin{equation}
\label{spaceV}
V_{N}^{(k+1)}=\Big\{w\in L^{Q}_{I_{\delta_{N}^{(k+1)}}}L^{r}:\|w\|_{L^{Q}_{I_{\delta_{N}^{(k+1)}}}L^{r}}\leq\frac{3c_{0}c_{E}c_{I}(\frac14)^{\frac1{Q}}(1+\frac{k+2}{4})^{\frac12-\frac1{r}}}{N}\Big\}
\end{equation}
and the operator

\begin{equation}
\label{oper}
T^{(k+1)}w=e^{it\partial_{x}^{2}}\psi_{k+1}^{N}\pm i\int_{0}^{t}e^{i(t-\tau)\partial_{x}^{2}}G(v^{(k+1)},w)\ d\tau.
\end{equation}

Our claim is that $T^{(k+1)}(V_{N}^{(k+1)})\subset V_{N}^{(k+1)}$. Indeed, let $w\in V_{N}^{(k+1)}$, then for $t\in[0,\delta_{N}^{(k+1)}]$ we have

$$\|e^{it\partial_{x}^{2}}\psi_{k+1}^{N}\|_{r}\leq c_{E}\|e^{it\partial_{x}^{2}}\psi_{k+1}^{N}\|_{M_{r,\frac{r}{r-1}}}=c_{E}\|e^{i(t+\sum_{i=0}^{k}\delta_{N}^{(i)})\partial_{x}^{2}}\psi_{0}^{N}\|_{M_{r,\frac{r}{r-1}}},$$
which is bounded above by

$$c_{E}c_{I}\Big(1+|t+\sum_{i=0}^{k}\delta_{N}^{(i)}|\Big)^{\frac12-\frac1{r}}\|\psi_{0}^{N}\|_{M_{r,\frac{r}{r-1}}}\leq \frac{c_{0}c_{E}c_{I}}{N}\Big(1+\sum_{i=0}^{k+1}\delta_{N}^{(i)}\Big)^{\frac12-\frac1{r}}\leq\frac{c_{0}c_{E}c_{I}}{N}\Big(1+\frac{k+2}{4}\Big)^{\frac12-\frac1{r}}.$$
From which it follows that

$$\|e^{it\partial_{x}^{2}}\psi_{k+1}^{N}\|_{L^{Q}_{I_{\delta_{N}^{(k+1)}}}L^{r}}\leq\frac{c_{0}c_{E}c_{I}}{N}\Big(1+\frac{k+2}{4}\Big)^{\frac12-\frac1{r}}\Big(\frac14\Big)^{\frac1{Q}}.$$
For the convolution integral part of the operator $T^{(k+1)}$ we use Proposition \ref{Tsu1} with the functions $v=v^{(k+1)}$, $w_{1}=w$ and $w_{2}=0$, to estimate the norm 

$$\Big\|\int_{0}^{t}e^{i(t-\tau)\partial_{x}^{2}}G(v^{(k+1)},w)\ d\tau\Big\|_{L^{Q}_{I_{\delta_{N}^{(k+1)}}}L^{r}}.$$
Thus, we have the upper bound

$$c_{Q}(\delta_{N}^{(k+1)})^{\frac12}\|v^{(k+1)}\|_{L^{\frac{4r}{r-2}}_{I_{\delta_{N}^{(k+1)}}}L^{r}}^{2}\|w\|_{L^{Q}_{I_{\delta_{N}^{(k+1)}}}L^{r}}+c_{Q}(\delta_{N}^{(k+1)})^{\frac{rQ-Q-2r}{rQ}}\|w\|_{L^{Q}_{I_{\delta_{N}^{(k+1)}}}L^{r}}^{3}$$
and this quantity in its turn is bounded from above by 

$$\frac{12c_{Q}k_{1}^{2}c_{0}^{3}c_{E}c_{I}(\frac14)^{\frac1{Q}}(1+\frac{k+2}{4})^{\frac12-\frac1{r}}}{\sqrt{M_{k+1}}}\cdot\frac1{N}+\frac{27c_{Q}c_{0}^{3}c_{E}^{3}c_{I}^{3}(\frac14)^{\frac3{Q}}(1+\frac{k+2}{4})^{\frac32-\frac3{r}}}{M_{k+1}^{\frac{rQ-Q-2r}{rQ}}}\cdot\frac1{N^{3+4\alpha(\frac{rQ-Q-2r}{rQ})}}.$$
From the choice of the $M_{k+1}$, see inequality (\ref{second}), we get that this is less than

$$2\cdot\frac{c_{0}c_{E}c_{I}}{N}\Big(1+\frac{k+2}{4}\Big)^{\frac12-\frac1{r}}\Big(\frac14\Big)^{\frac1{Q}}$$
which shows that $T^{(k+1)}$ maps the space $V_{N}^{(k+1)}$ into itself. 

Our next step is to show that for any $w_{1}, w_{2}\in V_{N}^{(k+1)}$ we have the contraction property

$$\|T^{(k+1)}(w_{1})-T^{(k+1)}(w_{2})\|_{L^{Q}_{I_{\delta_{N}^{(k+1)}}}L^{r}}<\frac23\|w_{1}-w_{2}\|_{L^{Q}_{I_{\delta_{N}^{(k+1)}}}L^{r}}.$$
The calculations are similar to the ones we just presented. They follow from Proposition \ref{Tsu1} and from the fact that $N$ can be chosen to be larger than $2$. By the Banach contraction mapping principle we immediately obtain a solution of (\ref{wd}) that lies in the space $V_{N}^{(k+1)}$ and is defined for $t\in[0,\delta_{N}^{(k+1)}]$. 

What remains is to estimate the quantity

$$\|w^{(k+1)}(t)-e^{it\partial_{x}^{2}}\psi_{k+1}^{N}\|_{L^{\infty}_{I_{\delta_{N}^{(k+1)}}}L^{2}}.$$
For this we use Proposition \ref{Tsu2} and get the upper bound

$$c_{Q}\Big[(\delta_{N}^{(k+1)})^{\frac{3Qr-4r-2Q}{4Qr}}\|v\|^{2}_{L^{\frac{4r}{r-2}}_{I_{\delta_{N}^{(k+1)}}}L^{r}}+(\delta_{N}^{(k+1)})^{\frac{5Qr-6Q-12r}{4Qr}}\|w\|^{3}_{L^{Q}_{I_{\delta_{N}^{(k+1)}}}L^{r}}\Big].$$
Substituting, we are able to bound this quantity from above by the sum of the following two expressions

$$\frac{12c_{Q}k_{1}^{2}c_{0}^{3}c_{E}c_{I}(\frac14)^{\frac1{Q}}(1+\frac{k+2}{4})^{\frac12-\frac1{r}}}{M_{k+1}^{\frac{3Qr-4r-2Q}{4Qr}}}\cdot\frac1{N^{1+\alpha\cdot\frac{Qr-4r-2Q}{Qr}}}$$

$$\frac{27c_{Q}c_{0}^{3}c_{E}^{3}c_{I}^{3}(\frac14)^{\frac3{Q}}(1+\frac{k+2}{4})^{\frac32-\frac3{r}}}{M_{k+1}^{\frac{5Qr-6Q-12r}{4Qr}}}\cdot\frac1{N^{3+\alpha\cdot\frac{5Qr-6Q-12r}{rQ}}}.$$
By the choice of $M_{k+1}$, see inequalities (\ref{third}) and (\ref{forth}), we get the desired inequality

\begin{equation}
\label{small}
\|w^{(k+1)}(t)-e^{it\partial_{x}^{2}}\psi_{k+1}^{N}\|_{L^{\infty}_{I_{\delta_{N}^{(k+1)}}}L^{2}}\leq 2\cdot N^{-1-\alpha\cdot\frac{Qr-4r-2q}{Qr}}.
\end{equation}
Finally, to prove that a global solution exists, it suffices to show that the following sum $\sum_{k=0}^{N^{1+\alpha\cdot\frac{2Qr-4r-2Q}{Qr}}}\delta_{N}^{(k)}$ diverges as $N\to\infty$. In other words 

\begin{equation}
\label{sumsum}
\lim_{N\to\infty}\frac1{N^{4\alpha}}\cdot\sum_{k=0}^{N^{1+\alpha\cdot\frac{2Qr-4r-2Q}{Qr}}}\frac1{(5+k)^{\frac{6(r-2)Q}{5Qr-6Q-12r}}}=\infty.
\end{equation}
By the use of the Euler-Maclaurin summation formula it is easy to see that the sum

$$\sum_{k=0}^{n}\frac1{(5+k)^{\beta}}$$
is asymptotic to $n^{1-\beta}$, for $0<\beta<1$. Therefore, if 

\begin{equation}
\label{fin}
\Big[1+\alpha\cdot\frac{2Qr-4r-2Q}{Qr}\Big]\Big[1-\frac{6(r-2)Q}{5Qr-6Q-12r}\Big]-4\alpha>0
\end{equation}
we are done. But this is equivalent to $(6r-r^{2})Q^{2}>2\alpha[(11r^{2}-19r+6)Q^{2}-14r^{2}Q-24r^{2}]$ which is exactly how $Q$ was chosen from the beginning of the proof. 

About the uniqueness assertion of the global solution it suffices to observe that for large $Q$ the space $L^{Q}_{I_{T}}L^{r}$ is a subspace of $L^{\frac{4r}{r-2}}_{I_{T}}L^{r}$ and a supremum type argument with Proposition \ref{Tsu1} immediately yield the desired result. 
\end{proof}

\textbf{Acknowledgments}: The authors gratefully acknowledge financial support by the Deut\-sche Forschungs\-gemeinschaft (DFG) through CRC 1173. D. H. also thanks the Alfried Krupp von Bohlen und Halbach Foundation for financial support.

\end{section}


\begin{thebibliography}{00}

\bibitem{BO}{\sc A. B\'enyi and K.A. Okoudjou}, {\em Local well-posedness of nonlinear dispersive equations on modulation spaces.} Bull. Lond. Math. Soc. 41, no. 3, 549–558 (2009). 
  
\bibitem{BOU}{\sc J. Bourgain}, {\em Global wellposedness of defocusing critical nonlinear Schr\"odinger equation in the radial case.} J. Amer. Math. Soc., 12(1):145-171, (1999).

\bibitem{FEI}{\sc Feichtinger}, {\em Modulation spaces on locally compact Abelian group.} Technical Report, University of Vienna, 1983, in: Proc. Internat. Conf. on Wavelet and applications, 2002, New Delhi Allied Publishers, India (2003), 99-140.

\bibitem{HTT}{\sc R. Hyakuna, T. Tanaka and M. Tsutsumi}, {\em On the global well-posedness for the nonlinear Schr\"odinger equations with large initial data of infinite $L^{2}$ norm.} Nonlinear Analysis: Theory, Methods and Applications, 74 (2011), 1304-1319.

\bibitem{HYTSU}{\sc R. Hyakuna and M. Tsutsumi}, {\em On existence of global solutions of Schr\"odinger equations with subcritical nonlinearity for $\hat{L^{p}}$ initial data.} Proc. of the Amer. Math. Soc. Volume 140, Number 11, November 2012, pages 3905-3920, S 0002-9939(2012)11314-0.

\bibitem{KA}{\sc T. Kato}, {\em An $L^{q,r}$-theory for nonlinear Schr\"odinger equations.} Adv. Stud. Pure Math., vol. 23 (1994), 223-238. 

\bibitem{KAT}{\sc T.Kato}, {\em The global Cauchy problems for the nonlinear dispersive equations on modulation spaces.} J. Math. Anal. Appl. 413 (2014), 821-840.

\bibitem{LUN}{\sc A. Lunardi}, {\em Interpolation Theory.} Lecture Notes (Scuola Normale Superiore), Vol. 9 (2009).

\bibitem{RUU}{\sc M. Ruzhansky, M. Sugimoto and B. Wang}, {\em Modulation Spaces and Nonlinear Evolution Equations.} Progress in Mathematics, Vol. 301 (2012), 267-283. 

\bibitem{Tsu}{\sc Y. Tsutsumi}, {\em $L^{2}$ solutions for nonlinear Schr\"odinger equations and nonlinear groups.} Funkcial. Ekvac. 30 (1987), 115-125.

\bibitem{VV}{\sc A. Vargas and L. Vega}, {\em Global wellposedness for $1$d nonlinear Schr\"odinger equation for data with an infinite $L^{2}$ norm.} J. Math. Pures Appl. 80, 10(2001), 1029-1044.

\bibitem{BH}{\sc B. Wang and H. Hudzik}, {\em The global Cauchy problem for the NLS and NLKG with small rough data.} J. Differ. Equations 232 (2007), 36-73.

\bibitem{WZCZ}{\sc B. Wang, Z. Huo, C. Hao and Z. Guo}, {\em Harmonic Analysis Method for Nonlinear Evolution Equations, I.} World Scientific, (2011).

\end{thebibliography}
\end{document}